\documentclass[a4paper]{article}
\usepackage[top=3.5cm,bottom=3.5cm,right=3cm,left=3cm]{geometry}
\usepackage[T1]{fontenc}

\usepackage{amsfonts}
\usepackage{amsmath}
\usepackage{amssymb}
\usepackage{amscd}
\usepackage{amsthm}
\usepackage{mathrsfs}

\usepackage[all]{xy}

\def\Z{{\mathbb{Z}}}

\newtheorem{theorem}{Theorem}[subsection]
\newtheorem{lemma}{Lemma}[subsection]
\newtheorem{proposition}{Proposition}[subsection]

\newtheorem{cor}{Corollary}

\newtheorem{defi}{Definition}

\begin{document}

\date{}
\title{\large{A FORMAL APPROACH  ``À LA NEUKIRCH'' OF   $\ell$-ADIC CLASS FIELD THEORY}}

\author{{\bf Stéphanie Reglade}}

\maketitle

{\footnotesize \noindent\textbf{Abstract}: \textit{Neukirch developed abstract class field theory in his famous book ``Class Field Theory''. We show that it is possible to derive Jaulent's $\ell$-adic class field from Neukirch's framework. The proof requires in both cases (local case and global case) to define suitable degree maps,  $G$-modules,  valuations and to prove the class field axiom.}}
\medskip

\medskip


\tableofcontents

\bigskip\bigskip

\noindent\textbf{\large Introduction}: \medskip

\noindent The $\ell$-adic class field theory, developed by Jaulent \cite{Ja1}, claims, in the local case, the existence of an 
isomorphism between the Galois group of the maximal and abelian pro-$\ell$-extension of a finite extension $K_{\mathfrak{p}}$ of $\mathbb{Q}_{\mathfrak{p}}$ and the $\ell$-adification of the multiplicative group of this local field; in the global case the existence of an isomorphism between the Galois group of the maximal and abelian pro-$\ell$-extension of a number field $K$ and the $\ell$-adification of the group of ideles. 

\noindent{Our goal in this paper is to rederive this theory, following Neukirch's abstract framework. It requires to define the degree map, the $G$-module and the valuation in the local and in the global case. We will have to check that the valuations are henselian with respect to the degree map, and to prove in each case the class field axiom. }

\noindent{We start with the local case in \S 2: we define suitable cohomology groups $\textrm{H}^{i}_{\ell}(G,V)$ in \S 1. After recalling the  keypoints of Neukirch's abstract therory in \S 2.1, we define} the group $G$, the $G$-module,the degree map $ \textrm{deg}: G \mapsto \mathbb{Z}_{\ell}$ and the valuation 
in  \S 2.2 to 2.4. Our main result  is the class field axiom:

\noindent{\textbf{Theorem}}
\noindent{\emph{For all cyclic $\ell$-extension $L_{\mathfrak{P}}$ of a local field $K_{\mathfrak{p}}$ we have
$$  |\textrm{H}^i_{\ell}(G(L_{\mathfrak{P}} / K_{p}),  \mathcal{R}_{L_{\mathfrak{P}}})|= \left\{
    \begin{array}{ll}
     [L_{\mathfrak{P}}:K_{p}] & \mbox{ for } i=0  \\
      1 & \mbox{ for } i=1
    \end{array}
\right.$$}}

\noindent{We then treat  the global case, our main result is: }
\smallskip

\noindent{\textbf{Theorem}}
\noindent{\emph{Let $L/K$ be a cyclic $\ell$-extension of algebraic number fields then we have:
\begin{center}
$  |\textrm{H}^{i}_{\ell}(G(L / K),  \mathcal{C}_{L})|= \left\{
    \begin{array}{ll}
     [L:K] & \mbox{ for } i=0  \\
      1 & \mbox{ for } i=1
    \end{array}
\right.$
\end{center}
}}
\smallskip

\noindent{The proof requires first to compute the Herbrand quotient of the idele class $\ell$-group (theorem 3.2.1). }We again define $G$, the $G$-module, $ \textrm{deg}: G \mapsto \mathbb{Z}_{\ell}$ and the valuation 
in \& 3.3 to 3.5.

\section{Preliminary}

\subsection{Notations}

\noindent{In the following $\ell$ is a fixxed rational prime number. Let's introduce the notations: }
\smallskip

\noindent{\textit{For a local field $K_{\mathfrak{p}}$ with maximal ideal $\mathfrak{p}$ and uniformizer $\pi_{\mathfrak{p}}$, we let}

$\mathcal{R}_{K_{\mathfrak{p}} }=\varprojlim_{k}  K_{\mathfrak{p}}^{\times} \diagup {  K_{\mathfrak{p}}^{\times  \ell^{k}}}$:  
the $\ell$-adification of the multiplicative group of a local field
\smallskip

$\mathcal{U}_{K_{\mathfrak{p}} }=\varprojlim_{k} {U}_{\mathfrak{p}}  \diagup  U_{\mathfrak{p}}^{\ell^k}$: the $\ell$-adification  of the group of units $U_{\mathfrak{p}} $  of $K_{\mathfrak{p}}$
\smallskip

$U_{\mathfrak{p}}^{1} $: the group of principal units of  $K_{\mathfrak{p}}$
\smallskip

$\mu_{\mathfrak{p}}^{0}$: the subgroup of $ U_{\mathfrak{p}}$, whose order is  finite and prime to $\mathfrak{p}$
\smallskip

$\mu_{\mathfrak{p}}$: the $\ell $- Sylow subgroup of $\mu_{\mathfrak{p}}^{0}$

\medskip

\noindent{\textit{For a number field $K$ we define }}

$\mathcal{R}_{K}=\mathbb{Z}_{\ell} \otimes_{\mathbb{Z}} K^{\times}$ : the
$\ell$-adic group of principal ideles
\smallskip

 $\mathcal{J}_{K} =\prod_{ \mathfrak{p} \in Pl_{K} } ^{res} \mathcal{R}_{K_{\mathfrak{p}}}$ : the $\ell$-adic idele group
\smallskip

$\mathcal{U}_{K}=\prod_{ \mathfrak{p} \in Pl_{K} } \mathcal{U}_{K_{\mathfrak{p}}}$ 
: the subgroup of units
\smallskip

$\mathcal{C}_{K}= \mathcal{J}_{K} /  \mathcal{R}_{K}$ : the $\ell$-adic idele class group 
\smallskip

\setlength{\parskip}{5pt plus2pt minus1pt}

\subsection{The $\mathbb Z_\ell$-cohomology}

\noindent{We use the following cohomology for $\mathbb{Z}_{\ell}$-modules. }

\begin{defi}
Let $F_n \longrightarrow \cdots \longrightarrow F_{0} \longrightarrow \mathbb{Z_{\ell}} \longrightarrow 0$ be a projective resolution of $\mathbb{Z}_{\ell}[G]$-modules, where $G$ is a $\ell$-group. Applying the functor $\textrm{Hom}_{G}(.,\mathbb{Z}_{\ell} \otimes A)$ we obtain:
\begin{displaymath}
0 \longrightarrow\textrm{Hom}_{G}(\mathbb{Z}_{\ell},\mathbb{Z}_{\ell} \otimes A) \longrightarrow\textrm{Hom}_{G}( F_{0},\mathbb{Z}_{\ell} \otimes A) \longrightarrow\textrm{Hom}_{G}( F_{1},\mathbb{Z}_{\ell} \otimes A) \longrightarrow \cdots 
\end{displaymath}
\begin{displaymath}
\cdots\longrightarrow\textrm{Hom}_{G}(F_{n-1},\mathbb{Z}_{\ell} \otimes A) \overset{\delta^{'}_{n-1}}{\longrightarrow}\textrm{Hom}_{G}( F_{n},\mathbb{Z}_{\ell} \otimes A)  \overset{\delta^{'}_{n}}{\longrightarrow}\textrm{Hom}_{G}( F_{n+1},\mathbb{Z}_{\ell} \otimes A) \cdots
\end{displaymath}
We denote $\textrm{H}^{n}_{\ell}(G,\mathbb{Z}_{\ell} \otimes A):= \textrm{Ker}\delta^{'}_{n} / \textrm{Im}\delta^{'}_{n-1}$.
\end{defi}

\begin{theorem}
If $G$ is a  $\ell$-group, and $A$ a $G$-module then: 
\begin{center}
$ \textrm{H}^{i}_{\ell}(G, \mathbb{Z}_{\ell} \otimes A)= \mathbb{Z}_{\ell} \otimes \textrm{H}^{i}(G,A)$
\end{center}
\end{theorem}

\begin{proof}

\noindent{We start with the projective resolution of free $F_{i}$ $\mathbb{Z}[G]$-modules:}

$F_n \longrightarrow F_{n-1} \longrightarrow \cdots \longrightarrow F_{1} \longrightarrow F_{0} \longrightarrow \mathbb{Z}
\longrightarrow 0.$

\noindent{i)Applying the functor $\textrm{Hom}_{G}(.,A)$ we get: }
$$ 0 \longrightarrow \textrm{Hom}_{G}(\mathbb{Z},A)\longrightarrow \cdots 
\longrightarrow \textrm{Hom}_{G}(F_{n-1},A)  \overset{\delta_{n-1}}\longrightarrow \textrm{Hom}_{G}(F_{n},A)  \overset{\delta_{n}}\longrightarrow \textrm{Hom}_{G}(F_{n+1},A)  \longrightarrow \cdots 
$$ and $$\mathrm{H}^{n}(G,A):= \mathrm{Ker}\delta_{n} / \mathrm{Im}\delta_{n-1}.$$

\noindent{ii)Since $\mathbb{Z}_{\ell}$ is a flat module, we obtain: }

$ \mathbb{Z}_{\ell} \otimes F_n \longrightarrow \mathbb{Z}_{\ell} \otimes F_{n-1} \longrightarrow \cdots \longrightarrow  \mathbb{Z}_{\ell} \otimes F_{0} \longrightarrow \mathbb{Z}_{\ell}
\longrightarrow 0$.

\noindent{ Applying now the functor $\textrm{Hom}_{G}(.,\mathbb{Z}_{\ell} \otimes A)$  we get: }
$$0 \longrightarrow\textrm{Hom}_{G}(\mathbb{Z}_{\ell},\mathbb{Z}_{\ell} \otimes A) \longrightarrow \cdots 
 \longrightarrow \textrm{Hom}_{G}(\mathbb{Z}_{\ell} \otimes F_{n-1},\mathbb{Z}_{\ell} \otimes A) \overset{\delta^{'}_{n-1}}{\longrightarrow}\textrm{Hom}_{G}(\mathbb{Z}_{\ell} \otimes F_{n},\mathbb{Z}_{\ell} \otimes A) \overset{\delta^{'}_{n}}{\longrightarrow} \cdots $$
and $$\textrm{H}^{n}_{\ell}(G,\mathbb{Z}_{\ell} \otimes A):= \textrm{Ker}\delta^{'}_{n} / \textrm{Im}\delta^{'}_{n-1}.$$

\noindent{iii) We now show that $\textrm{Hom}_{G}(\mathbb{Z}_{\ell} \otimes F_{i},\mathbb{Z}_{\ell} \otimes A)= \mathbb{Z}_{\ell} \otimes\textrm{Hom}_{G}(G,A).$}

\noindent{The $F_{i}$ are free $\mathbb{Z}[G]$-modules, using the additivity of the functor}
$\textrm{Hom}_{G}(.,A)$ it suffices to check the property on $\mathbb{Z}[G]$. But 
\begin{center}
$ \textrm{Hom}_{G}(\mathbb{Z}[G],A) \simeq A \quad \textrm{and} \quad \textrm{Hom}_{G}(\mathbb{Z}_{\ell}[G], \mathbb{Z}_{\ell} \otimes A) \simeq \mathbb{Z}_{\ell} \otimes A.$
\end{center}

\noindent{iv) We now show that:
$\textrm{Ker}\delta^{'}_{n}=\mathbb{Z}_{\ell} \otimes  \textrm{Ker}\delta_{n} \quad  \textrm{and} \quad \textrm{Im}\delta^{'}_{n-1}= \mathbb{Z}_{\ell} \otimes  \textrm{Im}\delta_{n-1}.$}

\noindent{Indeed, given a $\mathbb{Z}$-linear map $u: \; M \longrightarrow N$ and the corresponding $u^{'}: \, \mathbb{Z}_{\ell} \otimes M \longrightarrow 
\mathbb{Z}_{\ell} \otimes N$;  we have, since $\mathbb{Z}_{\ell}$ is flat, the exact sequence: }
$$0 \longrightarrow \mathbb{Z}_{\ell} \otimes \textrm{Ker}(u) \longrightarrow  \mathbb{Z}_{\ell} \otimes M \longrightarrow
 \mathbb{Z}_{\ell} \otimes \textrm{Im}(u) \longrightarrow 0. $$
Usually we also have $\textrm{Im}(u) \subset N$ and $ \mathbb{Z}_{\ell} \otimes \textrm{Im}(u) \subset \mathbb{Z}_{\ell} \otimes N$ by flatness of $\mathbb{Z}_{\ell}$. Finally, $$\textrm{Ker}\delta^{'}_{n} /\textrm{Im}\delta^{'}_{n-1} \simeq \mathbb{Z}_{\ell} \otimes (\textrm{Ker}\delta_{n} / \textrm{Im}\delta_{n-1}).$$
\end{proof}

\begin{cor}
Let $G$ be a finite cyclic $\ell$-group, and $A$ a $G$-module then the Herbrand quotient $h_{\ell}(G, A) := \frac{\textrm{H}^{}_{\ell}(G, \mathbb{Z}_{\ell} \otimes A)} {\textrm{H}^{-1}_{\ell}(G, \mathbb{Z}_{\ell} \otimes A)}$ satisfies:

i) if $A$ is a finite $G$-module, $h_{\ell}(G,A)=1$.

ii) if we have an exact sequence of $G$-modules: $0 \longrightarrow A \longrightarrow B \longrightarrow C \longrightarrow 0$ then 
$h_{\ell}(G,B)=h_{\ell}(G,A) \cdot h_{\ell}(G,C)$.

iii) $\textrm{H}^{1}_{\ell}(G, \mathbb{Z}_{\ell} \otimes A) \simeq \textrm{H}^{-1}_{\ell}(G, \mathbb{Z}_{\ell} \otimes A)$.

\end{cor}

\section{Local $\ell$-adic class field theory}

\subsection{Framework}

 The fundamental local $\ell$-adic theorem is: 
\begin{theorem}
 ~\cite[theorem 2.1]{Ja1}
Given a local field $ K_{\mathfrak{p}}/\mathbb{Q}_{\mathfrak{p}}$, the reciprocity map induces an isomorphism of topological $ \Z_{\ell}$-modules between $ \mathcal{R}_{K_{\mathfrak{p}}} = \varprojlim_{k} K_{\mathfrak{p}}^{\times} / K_{\mathfrak{p}}^{\ell^k} $ and the Galois group $ \mathcal{D}_{\mathfrak{p}} = \textrm{Gal}( K_{\mathfrak{p}}^{ab}/ K_{\mathfrak{p}})$ of the maximal and abelian pro-$\ell$-extension of $ K_{\mathfrak{p}}$.
Trough this isomorphism, the image of the inertia sub-group $\mathcal{I}_{\mathfrak{p}}$ is the sub-group of units $ \mathcal{U}_{K_{\mathfrak{p}}}$ of $ \mathcal{R}_{K_{\mathfrak{p}}}$.
The reciprocity map induces a one to one correspondence between closed sub-modules of $ \mathcal{R}_{K{\mathfrak{p}}}$ and abelian $\ell$-extensions of $ K_{\mathfrak{p}}$: in this correspondence, finite abelian $\ell$-extensions are associated to closed sub-modules of finite index of $\mathcal{R}_{K_{\mathfrak{p}}}$; it means to open sub-modules of $\mathcal{R}_{K_{\mathfrak{p}}}$.
\end{theorem}

Our purpose is to prove the existence of the local reciprocity map using Neukirch's abstract class field theory, which we now briefly recall ~\cite[p.~18-36]{Ne1}. We consider the following general framework: $G$ is an abstract profinite group, whose closed subgroups are denoted by $G_K$, those indices $K$ are called ``fields''. $G$ is equipped with a continuous and surjective homomorphism
$\textrm{deg}: \; G \longrightarrow \widehat{\mathbb{Z}} $.

\begin{enumerate}
\item We denote by $k$ the field such that $G_k=G$.

\item We denote by $\bar{k}$ the field such that $G_{\bar{k}}=\{1\}$.

\item If $G_L \subset G_K$, we write $K \subset L$.

\item L/K is said finite if $G_L$ is open ( closed of finite index) in $G_K$; the degree $[L:K]$ is then defined by $[L:K]=(G_K:G_L)$.

\item We write $ K=\prod K_i$ for $G_K= \cap_{i} G_{K_i}$.

\item We write $K=\cap K_i$ if $G_K$ is topologically generated by the $G_{K_i}$.

\item If $G_L$ is normal in $G_K$ we say that $L/K$ is \emph{a Galois extension} and we write $\textrm{Gal}(L/K):=G_K/G_L$.

\item The kernel of $\textrm{deg}$ is a subgroup of $G$ denoted by $G_{\tilde k}=I$ such that $ G/G_{\tilde k} \simeq \widehat{\mathbb{Z}}$. We can restrict $\textrm{deg}$ to $G_K$ and define:
$$ f_K=(\mathbb{Z}: \textrm{deg}(G_K)) \qquad e_K=(G_{\tilde k}:G_{\tilde K}) \qquad I_K=G_{\tilde K}.$$
If $L/K$ is an extension we put:
$$f_{L/K}=(\mathrm{deg}(G_K):\textrm{deg}(G_L)) \qquad e_{L/K}=(I_K:I_L).$$ 
They satisfy the following relations:
$$ f_{L/K}=f_L/f_K \qquad [L:K]=e_{L/K} \cdot f_{L/K}.$$

\item If $K$ is a finite extension $k$ we define $\tilde K= K \cdot \tilde k$.

\end{enumerate}

Neukirch's theory requires \textbf{a $G$-module and a henselian valuation with respect to $\textrm{ deg }$} ~\cite[p.~288]{Ne2}:  a multiplicative $G$-module $A$ is an abelian multiplicative group endowed with a continuous right action
\begin{center}
$\begin{array}{ccccc}
\sigma  & : & A & \to & A  \\
& & a & \mapsto & a^{\sigma}\\
\end{array}$
\end{center}
\noindent{i.e such that}
 $A = \bigcup_{[K:k]< \infty} A_K$,  where $A_K:= \{ a \in A \mid a^{\sigma}=a, \, \forall \sigma \in G_K \}=A^{G_K}$ and where $K$ runs through all finite extensions of $k$.

\noindent{This allows to define a new map, the norm map, which goes to the $G$-module $A_K$ in $A_k$:}
\begin{center}
 $N_{K/k}(a)=\prod_{\sigma}  a^{\sigma} $ 
\end{center}
where $\sigma$ runs through a representative coset of $G_{K}/G_{L}$. 

\noindent{A henselian valuation of $A_k$ with respect to $\textrm{deg}:\; G \rightarrow \widehat{\mathbb{Z}}$ is a homomorphism satisfying the following properties:s~\cite[p.~288]{Ne2}

($i$) $v(A_k)=Z $ such that $ \mathbb{Z} \subset  Z $  and $ Z/n \cdot Z \simeq  \mathbb{Z}/n \cdot \mathbb{Z} $ for all $n>0$

($ii$) $v(N_{K/k} A_K)=f_K.Z$ for all extensions $K$ of $k$.

\medskip

\noindent{Finally we introduce \emph{the class field axiom}: }

\noindent{\textbf{Axiom: }}
\emph{For all cyclic extension $L/K$, we have:
\begin{center}
$  |\textrm{H}^i(G(L/K), A_L)|= \left\{
    \begin{array}{ll}
     [L:K] & \mbox{ for } i=0  \\
      1 & \mbox{ for } i=-1
    \end{array}
\right.$
\end{center}}
\medskip

\noindent{In this context Neukirch proves the following fundamental theorem: ~\cite[p.~28]{Ne1}}
\begin{theorem}
Let $L/K$ be a finite Galois extension, $\sigma \in  G(L/K)^{ab}$, $\tilde \sigma \in \textrm{Gal}(\tilde L /K)$, (which is the Frobenius lift of $\sigma$) and $\Sigma$ i be the fixed field of $\tilde \sigma$, then the  homomorphism 
\begin{align*}
  r_{L/K}\colon   G(L/K)^{ab} &\longrightarrow A_K/N_{L/K}(A_L)    \\
           \sigma &\longmapsto  N_{\Sigma /K}(\pi_{\Sigma}) \; \mathrm{mod} \; N_{L/K}A_L 
\end{align*}
 is an isomorphism, where $\pi_{\Sigma}$ is a prime element of $A_{\Sigma}$.
\end{theorem}

We now define all necessary ingredients to obtain the main theorem of $\ell$-adic class field theory: theorem 2.5.1.

\subsection{$G$ and the $G$-module}

We consider the following context: 

. $k$ is a local field, (we use this notation instead of $k_{\mathfrak{p}}$).

. $k^{nr}$ is the maximal unramified pro-$\ell$-extension of $k$: the compositum of all unramified $\ell$-extensions.

. $\widehat{k}$  is the maximal  pro-$\ell$-extension of $k$: the compositum  of $\ell$-extensions of $k$.

\noindent {Classically $Gal(k^{nr}/k) \simeq \mathbb{Z}_{\ell}$ ~\cite[p.~41-42]{Ne1}}.

\noindent{We write }
\begin{center}
$G=\textrm{Gal}(\widehat{k}/k)$
\end{center}

\noindent{We consider the following $\mathbb{Z}$-module:}
$$A=\lim\limits_{\overrightarrow{L_{\mathfrak{P}}}} \mathcal{R}_{L_{\mathfrak{P}}}$$
where $L_{\mathfrak{P}}$ runs through all finite extensions of $K_{\mathfrak{p}}$,
and $\mathcal{R}_{L_{\mathfrak{P}}} =\varprojlim_{k}  L_{\mathfrak{P}}^{\times}   \diagup {L_{\mathfrak{P}}^{\times  \ell^{k}}}$. It is canonically identified to 
$$ A= \bigcup_{[L_{\mathfrak{P}}:K_{\mathfrak{p}}]< \infty}  \mathcal{R}_{L_{\mathfrak{P}}}.$$
If $L_{\mathfrak{P}}$  is a finite extension of $K_{\mathfrak{p}}$, 
$$A_{L_{\mathfrak{P}}}=\mathcal{R}_{L_{\mathfrak{P}}}$$
is a $\textrm{Gal}(L_{\mathfrak{P}}/K_{\mathfrak{p}})$ module. The group $G$ now axts on $A$ acts component by component.

\subsection{$\textrm{deg}: G \mapsto \mathbb{Z}_{\ell} $}

\begin{defi}
Let $\phi \in G$, its restriction to $k^{nr}$ defines an element of $\mathbb{Z}_{\ell}$, due to the isomorphism
$\textrm{Gal}(\widehat{k} / k ) \simeq \mathbb{Z}_{\ell}$. We define:
\begin{align*}
  \textrm{deg}  \colon  G&\longrightarrow \mathbb{Z}_{\ell}   \\
           \phi &\longmapsto  \phi_{\mid_{k^{nr}}}   
\end{align*}
$ \textrm{deg}$ is a surjective homomorphism whose kernel is $G_{k^{nr}}$ so that: $ G/G_{k^{nr}} \simeq \textrm{Gal}(k^{nr}/k) \simeq \mathbb{Z}_{\ell}$.
\end{defi}

\begin{defi}
Given a finite  $\ell$-extension $K$ of $k$ ,we define:
$$f_K:=(\mathbb{Z}_{\ell}:\mathrm{deg}(G_K)  \quad 
e_K:=(G_{k^{nr}}:I_K)$$
$$I_K=G_{K^{nr}} \cap G_{K}=G_{K \cdot k^{nr}}:=G_{K^{nr}}$$
\end{defi}

\begin{defi}
If $L/K$ is a finite $\ell$-extension we define:
\begin{center}
$f_{L/K}=(\textrm{deg}(G_K):\textrm{deg}(G_L)) \quad   e_{L/K}=(I_K:I_L)$ 
\end{center}
\end{defi}

\begin{proposition}
We have the following fundamental relations:
$$ f_{L/K}=f_L/f_K  \quad  e_{L/K} \cdot f_{L/K}=[L:K] $$
\end{proposition}

\begin{proof}
~\cite[p.~286]{Ne2} 
\end{proof}

\subsection{The valuation}

In $\ell$-adic class field theory, the degree is a homomorphism from $G$ to $\mathbb{Z}_{\ell}$, and the valuation $v$ is a homomorphism from $A_k$ to $\mathbb{Z}_{\ell}$. In this part, we denote  by $K_{\mathfrak{p}}$ a local field.

\noindent{For a finite extension $L_{\mathfrak{P}}$, we defined }
$$ A_{L_{\mathfrak{P}}}= \mathcal{R}_{L_{\mathfrak{P}}}=\varprojlim_{k}  L_{\mathfrak{P}}^{\times}   \diagup {  L_{\mathfrak{P}}^{\times  \ell^{k}}}$$ a $\textrm{Gal}(L_{\mathfrak{P}}/K_{\mathfrak{p}})$-module.
Jaulent proved that ~\cite[proposition 1.2]{Ja1}: 
$$\mathcal{R}_{L_{\mathfrak{P}}} \simeq  U_{\mathfrak{P}}^{1}  \cdot \pi_{\mathfrak{P}}^{\mathbb{Z}_{\ell}}
\; \mathrm{if} \, \mathfrak{P} \mid  \ell $$
$$\mathcal{R}_{L_{\mathfrak{P}}} \simeq {\mu_{\mathfrak{P}} } \cdot  {\pi_{\mathfrak{P}}^{\mathbb{Z}_{\ell}}} 
 \; \mathrm{if} \, \mathfrak{P} \nmid  \ell $$
This allows to define the valuation $v_{\mathfrak{P}}$  as giving the power in $\mathbb{Z}_{\ell}$ of the uniformising element.

\begin{proposition}
This valuation $v_{\mathfrak{P}}$  is henselian with respect to $ deg : G \mapsto \mathbb{Z}_{\ell}.$
\end{proposition}

\begin{proof}

The valuation associated to $\mathcal{R}_{K_{\mathfrak{p}}}$, $v_{\mathfrak{p}}$, 
is a surjective homomorphism; hence
\begin{center}
$v_{\mathfrak{p}}(\mathcal{R}_{K_{\mathfrak{p}}})=\mathbb{Z}_{\ell}$ := Z; and indeed $Z /n.Z \simeq 
\mathbb{Z} / n.\mathbb{Z}$  for all $n >0$.
\end{center}
We now check that $v_{\mathfrak{p}}(N_{L_{\mathfrak{P}}/K_{\mathfrak{p}}} \mathcal{R}_{K_{\mathfrak{p}}})=f_{ L_{\mathfrak{P}}/K_{\mathfrak{p}}} \cdot Z$.
The valuation $ v_L: \; \mathcal{R}_{L_{\mathfrak{P}}} \longrightarrow \mathbb{Z}_{\ell}$ can  be viewed as  an extension of the usual normalized valuation of $L_{\mathfrak{P}}$, denoted by $w_{\mathfrak{P}}$. 
In fact, we have the following commutative diagrams:

$$
\begin{CD}
  L_{\mathfrak{P}}^{\times} @>>> \mathcal{R}_{L_{\mathfrak{P}} }     & \qquad\qquad &          K_{\mathfrak{p}}^{\times} @>>>  \mathcal{R}_{K_{\mathfrak{p}} }\\
@V{w_{\mathfrak{P}}}VV   @VV{v_{\mathfrak{P}}}V     @V{w_{\mathfrak{p}}}VV   @VV{v_{ \mathfrak{p}}}V \\
\mathbb{Z}  @>>>  \mathbb{Z}_{\ell}       & &     \mathbb{Z} @>>>  \mathbb{Z}_{\ell}
\end{CD}
$$

The valuation $w_{\mathfrak{p}}$ extends uniquely to $L_{\mathfrak{P}}$ by:
$ \frac{1}{[L_{\mathfrak{P}}:K_{\mathfrak{p}}]} ( w_{\mathfrak{p}} \circ N_{L_{\mathfrak{P}/K_{\mathfrak{p}}}})$ and thus $v_{\mathfrak{p}}$ extends uniquely to $L_{\mathfrak{P}}$.
As $\frac{1}{e_{ L_{\mathfrak{P}}/K_{\mathfrak{p}}}} \cdot w_{\mathfrak{P}}$ is the continuation of $w_{\mathfrak{p}}$, we get:
$$\frac{1}{e_{ L_{\mathfrak{P}}/K_{\mathfrak{p}}} } \cdot v_{\mathfrak{P}}( \mathcal{R}_{L_{\mathfrak{P}} })=\frac{1}{[L_{\mathfrak{P}}:K_{\mathfrak{p}}]} \cdot v_{\mathfrak{p}}(N_{L_{\mathfrak{P}}/K_{\mathfrak{p}}} \mathcal{R}_{K_{\mathfrak{p}}})=\frac{1}{e_{ L_{\mathfrak{P}}/K_{\mathfrak{p}}}.f_{ L_{\mathfrak{P}}/K_{\mathfrak{p}}}} \cdot v_{\mathfrak{p}}(N_{L_{\mathfrak{P}}/K_{\mathfrak{p}}} \mathcal{R}_{k})$$
So we deduce that:
\begin{center}
$f_{ L_{\mathfrak{P}}/K_{\mathfrak{p}}} \cdot v_{\mathfrak{P}}( \mathcal{R}_{L_{\mathfrak{P}} })=v_k(N_{L_{\mathfrak{P}}/k}
\mathcal{R}_{k})$
\end{center}
Yet we have the relation $ f_{ L_{\mathfrak{P}}/K_{\mathfrak{p}}}=f_{ L_{\mathfrak{P}}} / f_{K_{\mathfrak{p}}}$, and due to the definition of $ f_{K_{\mathfrak{p}}}$ we have $ f_{K_{\mathfrak{p}}}=( \mathbb{Z}_{\ell}: d(G_{K_{\mathfrak{p}}}))=1$ as the degree is surjective.
Finally, we get:
\begin{center}
$ f_{ L_{\mathfrak{P}}/K_{\mathfrak{p}}} \cdot v_{\mathfrak{P}}( \mathcal{R}_{L_{\mathfrak{P}} })
= f_{ L_{\mathfrak{P}}/K_{\mathfrak{p}}} \cdot \mathbb{Z}_{\ell}=v_{\mathfrak{p}}(N_{L_{\mathfrak{P}}/K_{\mathfrak{p}}} \mathcal{R}_{K_{\mathfrak{p}}})$
\end{center}
for all finite extension $ L_{\mathfrak{P}}/K_{\mathfrak{p}}$ of $K_{\mathfrak{p}}$, the second point point ($ii$) is also checked.
\end{proof}

\subsection{The class field axiom}

We must show:

\begin{theorem}
For all cyclic $\ell$-extension $L_{\mathfrak{P}}$ of a local field $K_{\mathfrak{p}}$ we have
$$  |\textrm{H}^i_{\ell}(G(L_{\mathfrak{P}} / K_{p}),  \mathcal{R}_{L_{\mathfrak{P}}})|= \left\{
    \begin{array}{ll}
     [L_{\mathfrak{P}}:K_{p}] & \mbox{ for } i=0  \\
      1 & \mbox{ for } i=1
    \end{array}
\right.$$
\end{theorem}

\begin{proof}
Let $G:=G(L_{\mathfrak{P}} / K_{p})$

\noindent{We consider the following exact sequence: }
$$ 1 \longrightarrow L_{\mathfrak{P}}^{ \times div} \longrightarrow L_{\mathfrak{P}}^{\times} \longrightarrow  L_{\mathfrak{P}}^{\times}/ L_{\mathfrak{P}}^{\times div} \longrightarrow 1$$
where $ L_{\mathfrak{P}}^{ \times div}$ is the $\ell$-divisible part of $ L_{\mathfrak{P}}^{\times}$.
We recall that a multiplicative abelian group is said $\ell$-divisible if each element is a $\ell^{n}$-th power for an integer $n$.
Since $G$ is cyclic, we obtain the Herbrand hexagon: 

\begin{displaymath}
\xymatrix{
&\textrm{H}^{0}(G,L_{\mathfrak{P}}^{\times  div})  \ar [r] &  \textrm{H}^{0}(G, L_{\mathfrak{P}}^{\times})  \ar [dr]& \\
\textrm{H}^{-1}(G, L_{\mathfrak{P}}^{\times}/ L_{\mathfrak{P}}^{\times div}) \ar[ur] &  &  & \textrm{H}^{0}(G, L_{\mathfrak{P}}^{\times}/ L_{\mathfrak{P}}^{\times div}) \ar[dl] \\
&  \textrm{H}^{-1}(G,L_{\mathfrak{P}}^{\times })\ar[ul] &  \textrm{H}^{-1}(G,L_{\mathfrak{P}}^{\times div }) \ar[l] & \\
}
\end{displaymath}

i) Hilbert's  theorem 90 states that $ \textrm{H}^{-1}(G,L_{\mathfrak{P}}^{\times }) =1$.
\bigskip

ii)We  show that $\textrm{H}^{0}(G,L_{\mathfrak{P}}^{\times div })=1$  and  $\textrm{H}^{-1}(G,L_{\mathfrak{P}}^{\times div })=1$. By Hensel's lemma we have:
$L_{\mathfrak{P}}^{\times} \simeq \mu_{\mathfrak{P}}^{0} \cdot U_{\mathfrak{P}}^{1} \cdot \pi_{\mathfrak{P}}^{\mathbb{Z}}$
 \quad and \quad $ \mu_{\mathfrak{P}}^{0} \simeq  \mu_{\mathfrak{P}} \cdot   \mu_{\mathfrak{P}, div}$
where $ \mu_{\mathfrak{P}}$ is the $\ell$-Sylow subgroup of the group of roots of units and $ \mu_{\mathfrak{P}, div}$ is its $\ell$-divisible part.
\medskip

$\bullet$ \emph{case 1}: If $\mathfrak{P} \nmid \ell $ then $ U_{{\mathfrak{P}}}^{1}$ is a $\mathbb{Z}_{\mathfrak{P}}$-module, as  $\mathfrak{P}$ is invertible in $\mathbb{Z}_{\ell}$, so $ U_{{\mathfrak{P}}}^{1}$ is $\ell$ -divisible and so is $\mu_{\mathfrak{P}, div} \cdot U_{\mathfrak{P}}^{1}$. We have $h(G, \mu_{\mathfrak{P}, div} \cdot U_{\mathfrak{P}}^{1})=h(G, \mu_{\mathfrak{P}, div}) \cdot h(G, U_{\mathfrak{P}}^{1}) $; but $h(G,\mu_{\mathfrak{P}, div})=1$  (as $\mu_{\mathfrak{P}, div}$ is a finite group ) and  $ h(G,  U_{\mathfrak{P}}^{1})=1$ s~\cite[p.~40]{Ne1} so:
$h(G, \mu_{\mathfrak{P}, div} \cdot U_{\mathfrak{P}}^{1}):=h(G,L_{\mathfrak{P}}^{\times div })=1.$
Moreover, if $A$ is a $G$-module by definition $\textrm{H}^{0}(G,A)= \textrm{Ker}(\delta) / \textrm{Im}(\nu)$ where

\def\P{\mathfrak{P}}
\def\p{\mathfrak{p}}
\begin{align*}
  \delta\colon  A &\longrightarrow B     \qquad & \mu\colon A &\longrightarrow B \\
           a &\longmapsto (\sigma-1)a      &a &\longmapsto Tr_{L_\P/K_\p} (a)
\end{align*}

\noindent{If $a \in \textrm{Ker}(\delta) \cap L_{\mathfrak{P}}^{\times div}$ then $a \in (\mu_{\mathfrak{P}, div} \cdot U_{\mathfrak{P}}^{1})^{G}=(\mu_{\mathfrak{p}, div} \cdot U_{\mathfrak{p}}^{1})$ since the extension is  Galois, where}
 $ K_{\mathfrak{p}}^{\times} \simeq   \mu_{\mathfrak{p}} \cdot \mu_{\mathfrak{p}, div}\cdot U_{\mathfrak{p}}^{1} \cdot
\pi_{\mathfrak{p}}^{\mathbb{Z}}$. Consequently $a \in K_{\mathfrak{p}}^{\times div}$ and  so we can choose $b \in K_{\mathfrak{p}}^{\times}$ such that $ a=b^{\ell^{[L_{\mathfrak{P}}:K_{p}]}}=N(b)$. It follows that $ \textrm{H}^{0}(G,L_{\mathfrak{P}}^{ \times div})=1$, as
$h(G,L_{\mathfrak{P}}^{ \times div})=1$ and we finally get $ \textrm{H}^{-1}(G,L_{\mathfrak{P}}^{ \times div})=1$. 
\bigskip

$\bullet$ \emph{case 2}:
If $ \mathfrak{P} \mid \ell $ the group $\mu_{\mathfrak{P}}^{0}$ is $\ell$-divisible, and as the group of principal units is a noetherian  $\mathbb{Z_{\ell}}$-module, it is isomorphic to the inverse limit of its finite quotients:  
$L_{\mathfrak{P}}^{\times} \simeq \underbrace{\mu_{\mathfrak{P}}^{0}}_{div \, part} \cdot U_{\mathfrak{P}}^{1} \cdot \pi_{\mathfrak{P}}^{\mathbb{Z}}$. Since $\mu_{\mathfrak{P}}^{0}$ is  finite we have $ h(G,\mu_{\mathfrak{P}}^{0})=1$; using the same arguments as in case 1, we finally obtain that $ \textrm{H}^{0}(G,L_{\mathfrak{P}}^{\times div})$ is trivial and so is $ \textrm{H}^{-1}(G,L_{\mathfrak{P}}^{\times div})$.
\bigskip

iii)Using Herbrand's hexagon, we get $\textrm{H}^{-1}(G, L_{\mathfrak{P}}^{\times}/ L_{\mathfrak{P}}^{\times div})=1$.
\smallskip

iv) From Herbrand's hexagon we obtain
 $ \textrm{H}^{0}(G,L_{\mathfrak{P}}^{\times }) \simeq  \textrm{H}^{0}(G, L_{\mathfrak{P}}^{\times}/ L_{\mathfrak{P}}^{div}). $ But due to the local class field axiom, we have: $ | \textrm{H}^{0}(G,L_{\mathfrak{P}}^{\times })|=
[L_{\mathfrak{P}}:K_{p}]$. Finally, we get $ |  \textrm{H}^{0}(G, L_{\mathfrak{P}}^{\times}/ L_{\mathfrak{P}}^{div})|=
[L_{\mathfrak{P}}:K_{p}]$.
\smallskip

v) We show that  $h_{\ell}(G, \mathcal{R}_{L_{\mathfrak{P}}})=  [L_{\mathfrak{P}}:K_{p}]$.

\noindent{We now consider the following exact sequence, where $\mathbb{Z}_{\ell}$ is considered as a trivial $G$-module:}
$$1 \longrightarrow  \mathcal{U}_{L_{\mathfrak{P}}} \longrightarrow \mathcal{R}_{L_{\mathfrak{P}}}  \overset{v_{\mathfrak{P}}}{\longrightarrow} \mathbb{Z}_{\ell} \longrightarrow 1.$$
Recall that 
\begin{center}
:  if $\mathfrak{P} \mid  \ell $ \quad  $\mathcal{R}_{L_{\mathfrak{P}}} \simeq  U_{\mathfrak{P}}^{1}  \cdot \pi_{\mathfrak{P}}^{\mathbb{Z}_{\ell}} $ \quad  and \quad  ${\mathcal{U}_{L_{\mathfrak{P}} }} \simeq   U_{\mathfrak{P}}^{1}  $ 
\end{center}
\begin{center}
: else \quad $\mathcal{R}_{L_{\mathfrak{P}}} \simeq {\mu_{\mathfrak{P}} } \cdot  {\pi_{\mathfrak{P}}^{\mathbb{Z}_{\ell}}} $
\quad and \quad ${\mathcal{U}_{L_{\mathfrak{P}} }} \simeq \mu_{\mathfrak{P}}$ 
\end{center}
So,
$$ h_{\ell}(G, \mathcal{R}_{L_{\mathfrak{P}}})= h_{\ell}(G, \mathcal{U}_{L_{\mathfrak{P}}}) \cdot h_{\ell}(G, \mathbb{Z}_{\ell}). $$
Since $\mathbb{Z}_{\ell}$ is a trivial $G$-module we have:
\begin{center}
$ \textrm{H}^{0} (G,\mathbb{Z}_{\ell}) \simeq \mathbb{Z}_{\ell} / (|G| \cdot \mathbb{Z}_{\ell})$
\quad  $ \textrm{H}^{-1}(G,\mathbb{Z}_{\ell})=1$ \quad and \quad $h(G,\mathbb{Z}_{\ell})=[L_{\mathfrak{P}}:K_{\mathfrak{p}}]$.
\end{center}
Consequently it suffices to show that $ h_{\ell}(G,\mathcal{U}_{L_{\mathfrak{P}}})=1$.

\noindent{\emph{For $\mathfrak{P} \nmid  \ell $}}: as $\mu_{\mathfrak{P}, \ell}$ is the $ \ell $-Sylow subgroup of the group of units in $L_{\mathfrak{P}}$ it is a finite group, so a finite $G$-module; we use Herbrand's property, we get $h(G,\mathcal{U}_{L_{\mathfrak{P}}})=1$.
\smallskip

\noindent{\emph{For $\mathfrak{P} \mid  \ell $}}: we use  $h(G,U_{L_{\mathfrak{P}}})=1$  ~\cite[p.~40]{Ne1} and the exact sequence:
$$ 1 \longrightarrow U_{L_{\mathfrak{P}}}^{1}\longrightarrow U_{L_{\mathfrak{P}}} \longrightarrow U_{L_{\mathfrak{P}}}/
 U_{L_{\mathfrak{P}}}^{1} \longrightarrow 1$$
By Hensel's lemma $U_{L_{\mathfrak{P}}}/ U_{L_{\mathfrak{P}}}^{1} \simeq \kappa^{\ast}$ where $\kappa$ is the residue field. So $h(G, U_{L_{\mathfrak{P}}})= h(G, U_{L_{\mathfrak{P}}}^{1} ) \cdot h(G, U_{L_{\mathfrak{P}}}/ U_{L_{\mathfrak{P}}}^{1} )$. In this case we also obtain, $h(G, U_{L_{\mathfrak{P}}}^{1})=1$.

 In both cases, we have $ h_{\ell}(G,\mathcal{U}_{L_{\mathfrak{P}}})=1$. Finally  $h_{\ell}(G, \mathcal{R}_{L_{\mathfrak{P}}})=  [L_{\mathfrak{P}}:K_{p}]$.
\smallskip

vi) Hence, we have: 

$ |  \textrm{H}^{0}(G, L_{\mathfrak{P}}^{\times}/ L_{\mathfrak{P}}^{div})|=
[L_{\mathfrak{P}}:K_{p}]$, \qquad $ |  \textrm{H}^{-1}(G, L_{\mathfrak{P}}^{\times}/ L_{\mathfrak{P}}^{div})|=1$, \qquad 
$h_{\ell}(G, \mathcal{R}_{L_{\mathfrak{P}}})=  [L_{\mathfrak{P}}:K_{p}]$.

As  ${\mathcal{R}_{L_{\mathfrak{P}} }}=\varprojlim_{k}  L_{\mathfrak{P}}^{\times}   \diagup {  L_{\mathfrak{P}}^{\times  \ell^{k}}}=
\mathbb{Z}_{\ell} \otimes L_{\mathfrak{P}}^{\times}/L_{\mathfrak{P}}^{\times div} $  we get
$\textrm{H}^{0}(G, L_{\mathfrak{P}}^{\times}/L_{\mathfrak{P}}^{\times div})=\textrm{H}^{0}_{\ell}(G,\mathcal{R}_{L_{\mathfrak{P}} })$ and we obtain
$$| \mathrm{H}^{0}(G, L_{\mathfrak{P}}^{\times}/L_{\mathfrak{P}}^{\times div})|=|\mathrm{H}^{0}_{\ell}(G,\mathcal{R}_{L_{\mathfrak{P}} })|=[L_{\mathfrak{P}}:K_{p}]$$
But $ h(G, \mathcal{R}_{L_{\mathfrak{P}}})= [L_{\mathfrak{P}}:K_{p}]$ so we deduce:
$$\textrm{H}^{-1}_{\ell}(G,\mathcal{R}_{L_{\mathfrak{P}} })=1. $$
And since $G$ is cyclic, we obtain
$$\textrm{H}^{1}_{\ell}(G,\mathcal{R}_{L_{\mathfrak{P}} })=1. $$

\end{proof}

\begin{cor}
$(\textrm{\textrm{deg}},v)$ is a class field pair, and $A_K=\mathcal{R}_{K_{\mathfrak{p}}}$ satisfies the class field axiom. 
Thus for all Galois $\ell$-extension $L_{\mathfrak{P}} $ of a finite extension $K_{\mathfrak{p}}$ of $\mathbb{Q}_{\mathfrak{p}}$ we get an isomorphism:
$$\textrm{Gal}( L_{\mathfrak{P}}/K_{\mathfrak{p}})^{ab} \simeq \mathcal{R}_{K_{\mathfrak{p}}}/N_{ L_{\mathfrak{P}} /K_{\mathfrak{p}}}
\mathcal{R}_{L_{\mathfrak{P}}}.$$
In particular, we get a one to one correspondence between finite  abelian  $\ell$-extensions of a local field and the closed subgroups of finite index of $\mathcal{R}_{K_{\mathfrak{p}}}$.
\end{cor}

\section{Global $\ell$-adic class field theory}

\subsection{Introduction}

The fundamental global $\ell$-adic class field theory is the following:
\begin{theorem}
~\cite[theorem 2.3]{Ja1}
Given a number field $ K$, the reciprocity map induces a continuous isomorphism between  the $\ell$-group of ideles $ \mathcal{J}_K$ of $K$ and the Galois group $ G_K^{ab}=\textrm{Gal}(K^{ab}/K)$ of the maximal abelian pro-$\ell$-extension of $K$. The kernel of this morphism is the subgroup $\mathcal{R}_K$ of principal ideles. In this correspondence, the decomposition subgroup $\mathcal{D}_{\mathfrak{p}}$ of a prime $\mathfrak{p}$ of $K$ is the image in $G_K^{ab}$ of the sub-group $\mathcal{R}_{K_{\mathfrak{p}}}$ of $\mathcal{J}_K$; and the inertia sub-group $\mathcal{I}_{\mathfrak{p}}$ is the image of the subgroup of units $\mathcal{U}_{K_{\mathfrak{p}}}$ of $\mathcal{R}_{K_{\mathfrak{p}}}$. The reciprocity map leads to a one to one correspondence between closed sub-modules of $ \mathcal{J}_K$ containing $\mathcal{R}_K$ and abelian $\ell$-extensions of $ K$. Each sub-extension of $K^{ab}$ is the fixed field of a unique closed sub-module of  $\mathcal{J}_K$ containing $\mathcal{R}_K$. In this correspondence, finite and abelian $\ell$-extensions of $K$ are associated to closed sub-modules of finite index of $\mathcal{J}_K$ containing $\mathcal{R}_K$, it means to open sub-modules of $\mathcal{J}_K$ containing $\mathcal{R}_K$.
\end{theorem}
Our goal is to prove the existence of the reciprocity map in the global case using Neukirch's abstract theory. We now define all necessary ingredients to obtain the main theorem of $\ell$-adic class field theory: Theorem 2.

\subsection{The Herbrand quotient}

\begin{lemma}
Let $L/K$ be a finite extension of number fields, then the injection of  $\mathcal{J}_K$ in $\mathcal{J}_L$ induces an injection between their $\ell$-adic idele class
groups: $  \alpha \cdot \mathcal{R}_K   \longmapsto  \; \alpha \cdot \mathcal{R}_L$.
\end{lemma}

\begin{proof}
The injection of $\mathcal{J}_K$ in $\mathcal{J}_L$ maps $ \mathcal{R}_K$ to $\mathcal{R}_L$ thus the map is well-defined and yields  a homomorphism between $\mathcal{C}_K$ and $\mathcal{C}_L$. To show that this homomorphism is injective it suffices to prove that $\mathcal{J}_K \cap \mathcal{R}_L=\mathcal{R}_K$. Let $M/K$ be the Galois closure of $L/K$, with Galois group  $G$.
We have
\begin{center}
 $\mathcal{J}_K \subseteq \mathcal{J}_L \subseteq \mathcal{J}_M $ \quad and \quad $\mathcal{R}_K \subseteq \mathcal{R}_L \subseteq \mathcal{R}_M$ 
\end{center}
thus 
\begin{center}
$ \mathcal{J}_K \cap \mathcal{R}_L  \subseteq \mathcal{J}_K \cap \mathcal{R}_M  \subseteq   (\mathcal{J}_K \cap \mathcal{R}_M )^{G} \subseteq  \mathcal{J}_K \cap \mathcal{R}_M ^{G} = \mathcal{J}_K \cap \mathcal{R}_K=\mathcal{R}_K.$
\end{center}
\end{proof}

\begin{lemma}
Let $L/K$ be a finite Galois $\ell$-extension, $G$ its Galois group, then the $\ell$-adic idele class group  $\mathcal{C}_L$ of $L$  is canonically a $G$-module and $\mathcal{C}_L^{G}=\mathcal{C}_K.$
\end{lemma}

\begin{proof}
$\mathcal{J}_L$ is a  $G$-module which  contains $\mathcal{R}_L$ as a sub-$G$-module. The action $ ( \sigma, \alpha \cdot \mathcal{R}_L ) \mapsto  \sigma(\alpha)  \cdot \mathcal{R}_L$ endows  $\mathcal{C}_L$ with a $G$-module structure. As we have the exact sequence:
\begin{center}
$1 \longrightarrow \mathcal{R}_L \longrightarrow \mathcal{J}_L \longrightarrow \mathcal{C}_L \longrightarrow 1$
\end{center}
we obtain:
\begin{center}
$1 \longrightarrow \mathcal{R}_L^G \longrightarrow \mathcal{J}_L^G \longrightarrow \mathcal{C}_L^G  \longrightarrow \textrm{H}^{1}_{\ell}(G,\mathcal{R}_L).$
\end{center}

\noindent{But }$ \mathcal{R}_L^{G}={(\mathbb{Z}_{\ell} \otimes L^{\times})}^{G}=\mathbb{Z}_{\ell} \otimes {( L^{\times})}^{G}= \mathcal{R}_{K}$ and $\mathcal{J}_L^G=\mathcal{J}_K$. Theorem 1.0.1 and  Hilbert's theorem 90 imply $ \textrm{H}^{1}_{\ell}(G,\mathcal{R}_L)=1$  and we are done.
\end{proof}
\smallskip

\begin{theorem} The Herbrand quotient of the $\ell$-adic idele class group. 

\noindent{Let $L/K$ be a Galois cyclic $\ell$-extension of finite degree $\ell^{n}$, $G$  its Galois group then we have}
\begin{center}
$ h_{\ell}(G,\mathcal{C}_L)= \dfrac{ |\textrm{H}_{\ell}^{0}(G,\mathcal{C}_L| }{|\textrm{H}_{\ell}^{1}(G,\mathcal{C}_L)|}=\ell^{n}$.
\end{center}
In particular $ (\mathcal{C}_K:N_{L/K}\mathcal{C}_L)  \ge \ell^n$.
\end{theorem}

\begin{proof}
The proof runs in four steps.

\noindent\textbf{Step 1:}

\noindent We show in this part that for $S$ a big enough set of primes we have: $$ \mathcal{J}_K=\mathcal{J}_K^S \cdot \mathcal{R}_K 
\quad \text{where} \quad  \mathcal{J}_{K}= \bigcup_{S} \mathcal{J}_{K}^{S}
\quad \text{and} \quad 
\mathcal{J}_{K}^{S}=\prod_{\mathfrak{p} \in S } (\mathcal{R}_{K_{\mathfrak{p}}} )\prod_{\mathfrak{p} \not\in S } (\mathcal{U}_{K_{\mathfrak{p}}})$$
where $S$ runs through finite sets of primes of $K$. 

\noindent{$\mathcal{D}_K:= \bigoplus_{\mathfrak{p} \nmid \infty} \mathfrak{p}^{\mathbb{Z}_{\ell}} \bigoplus_{\mathfrak{p} \mid \infty} \mathfrak{p}^{\mathbb{Z}_{\ell} / 2 \cdot \mathbb{Z}_{\ell}}.$  We consider the topological direct sum: $\mathcal{J}_K=\mathcal{D}_K \oplus \mathcal{U}_K $ and the map:}
\begin{align*}
  \phi\colon  \mathcal{J}_K  &\longrightarrow \mathcal{D}_K  \\
           \alpha = ( \alpha_{\mathfrak{p}} ) &\longmapsto  \prod_{\mathfrak{p} \nmid \infty} \mathfrak{p}^{v_{\mathfrak{p}}( \alpha_{\mathfrak{p}} )}  
\end{align*}
 This homomorphism is surjective and its kernel is $\mathcal{J}_K^{S_{\infty}}$, where $S_{\infty}=\{ \mathfrak{p} \mid \infty \}$. So we get the isomorphism: $ \mathcal{J}_K/\mathcal{J}_K^{S_{\infty}} \simeq \mathcal{D}_K.$
Let $\mathcal{P}_K$ be the  the image of $\mathcal{R}_K$ in $\mathcal{D}_K$, we get: 
$ \mathcal{R}_K \cdot \mathcal{J}_K^{S_{\infty}} /\mathcal{J}_K^{S_{\infty}} \simeq \mathcal{P}_K$
That is why:
$$\mathcal{J}_K /\mathcal{R}_K \cdot \mathcal{J}_K^{S_{\infty}} \simeq \mathcal{D}_K /\mathcal{P}_K \simeq \mathcal{C}{\ell}_K$$ where $ \mathcal{C}{\ell}_K$ is the class group of divisors, ~\cite[p.~364]{Ja1}.
In particular $\mathcal{D}_K /\mathcal{P}_K $ is finite.

\noindent{Let  $a_1, \, a_2, \ldots, \,a_h$  be representatives for classes in $\mathcal{D}_K /\mathcal{P}_K $;  let $\mathfrak{p}_1, \ldots ,\mathfrak{p}_l$ be the primes which divide  $a_1,a_2, \ldots, a_h$ and let  $S:= S_{\infty} \cup \{ \mathfrak{p}_1, \ldots ,\mathfrak{p}_l \}$}
Let $\alpha=(\alpha_{\mathfrak{p}}) \in \mathcal{J}_K$, we write $\phi(\alpha)=a_i \cdot d$ where $d \in \mathcal{R}_K$. 
Then $\alpha \cdot d^{-1} \in \mathcal {J}_{K}^S$.
\medskip

\noindent\textbf{Step 2: the cohomology of $\mathcal{J}_L$ and $\mathcal{J}_L^S$}

\noindent{We first define for $L/K$ a  finite Galois extension (whose Galois group is $G$): }
$$\mathcal{J}_L^{\mathfrak{p}}=\prod_{ \mathfrak{P} \mid \mathfrak{p} }  \mathcal{R}_{L_{\mathfrak{P}}},  \quad \mathcal{U}_L^{\mathfrak{p}}=\prod_{ \mathfrak{P} \mid \mathfrak{p} }  \mathcal{U}_{L_{\mathfrak{P}}}$$
for each prime $\mathfrak{p}$ of $K$.
As an element of $G$ permutes the primes over $\mathfrak{p}$, $\mathcal{J}_L^{\mathfrak{p}}$ and 
$\mathcal{U}_L^{\mathfrak{p}}$ are $G$-modules and  we have:
\begin{center}
$\mathcal{J}_L=\prod_{\mathfrak{p}  } \mathcal{J}_L^{\mathfrak{p}},
\quad
\mathcal{U}_L=\prod_{\mathfrak{p}} \mathcal{U}_L^{\mathfrak{p}}$
\end{center}
Let  $ \mathfrak{P} $ be  a fixed prime of $L$ over $\mathfrak{p}$, $ G_{\mathfrak{P}}= \textrm{Gal}(L_{\mathfrak{P}} / K_{\mathfrak{p}})      \subseteq  G$ the decomposition subgroup and  $\sigma$ run through the cosets $G/ G_{\mathfrak{P}}$ then:  $\sigma(\mathfrak{P})$ runs through the different primes of $L$ over $\mathfrak{p}$, and

$$ \mathcal{J}_L^{\mathfrak{p}}=\prod_{ \sigma \in G/ G_{\mathfrak{P}} }  \mathcal{R}_{L_{\sigma(\mathfrak{P}})}=
\prod_{ \sigma \in G/ G_{\mathfrak{P}} } \sigma(  \mathcal{R}_{L_{\mathfrak{P}}}), \quad
\mathcal{U}_L^{\mathfrak{p}}=\prod_{ \sigma \in G/ G_{\mathfrak{P}} }  \mathcal{U}_{L_{\sigma(\mathfrak{P}})}$$

\noindent{Thus  we deduce that }
$ \mathcal{J}_L^{\mathfrak{p}}$ et $ \mathcal{U}_L^{\mathfrak{p}}$ are induced $G$-modules and
$$ \mathcal{J}_L^{\mathfrak{p}}=\textrm{Ind}_{G_{\mathfrak{P}}}^G(\mathcal{R}_{L_{\mathfrak{P}}}), \quad
\mathcal{U}_L^{\mathfrak{p}}=\textrm{Ind}_ {G_{\mathfrak{P}}}^{G}(\mathcal{U}_{L_{\mathfrak{P}}}). $$

\noindent{We write for  $S$  a set of primes of $K$ : $J_L^{S}:=J_L^{\overline{S}}$, where $\overline{S}$ is the set of primes of $L$ over $S$. Then we have the decomposition of $G$-modules:
$$\mathcal{J}_{L}^{S}=\prod_{\mathfrak{p} \in S }( \prod_{ \mathfrak{P} \mid \mathfrak{p} }  \mathcal{R}_{L_{\mathfrak{P}}}) \prod_{\mathfrak{p} \not\in S } (\prod_{\mathfrak{P} \mid \mathfrak{p}} \mathcal{U}_{L_{\mathfrak{P}}}) 
=\prod_{\mathfrak{p} \in S } \mathcal{J}_L^{\mathfrak{p}} \cdot  \prod_{\mathfrak{p} \not\in S } \mathcal{U}_L^{\mathfrak{p}}.$$}
\medskip

\begin{proposition}
Let  $S$ be the set of primes containing the infinite and the ramified primes, let  $\mathfrak{P}$ be a prime of $L$ over $\mathfrak{p}$  and $G_{\mathfrak{P}}$ the decomposition sub-group; then for $i=0,1$ we have: 
$$  \textrm{H}^{i}_{\ell}(G,\mathcal{J}_L^{S}) \simeq \bigoplus_{\mathfrak{p} \in S} \textrm{H}^{i}_{\ell}(G_{\mathfrak{P}}, \mathcal{R}_{L_{\mathfrak{P}}}) \quad  \textrm{and} \quad \textrm{H}^{i}_{\ell}(G,\mathcal{J}_L) \simeq \bigoplus_{\mathfrak{p}} \textrm{H}^{i}_{\ell}(G_{\mathfrak{P}},\mathcal{R}_{L_{\mathfrak{P}}}) $$
\end{proposition}

\begin{proof}
We have $ \mathcal{J}_L^S= \bigoplus_{\mathfrak{p} \in S} \mathcal{J}_L^{\mathfrak{p}} \oplus V \quad
\mathrm{where} \quad  V=\prod_{\mathfrak{p} \not\in S } \mathcal{U}_L^{\mathfrak{p}}.$
That is why we obtain  the isomorphism  for $i=0,1$:
$$\textrm{H}^{i}_{\ell}(G,\mathcal{J}_L)= \bigoplus_{\mathfrak{p} \in S} \textrm{H}^{i}_{\ell}(G,\mathcal{J}_L^{\mathfrak{p}}) \oplus \textrm{H}^{i}_{\ell}(G,V) \quad 
\textrm{and the injection} \quad \textrm{H}^{i}_{\ell}(G,V) \longrightarrow \prod_{\mathfrak{p} \not\in S} \textrm{H}^{i}_{\ell}(G,\mathcal{U}_L^{\mathfrak{p}}).$$
Moreover by the previous proposition $\mathcal{J}_L^{\mathfrak{p}}$ and $\mathcal{U}_L^{\mathfrak{p}}$ are induced $G$-modules, so
$$\textrm{H}^{i}_{\ell}(G,\mathcal{J}_L^{\mathfrak{p}}) \simeq  \textrm{H}^{i}_{\ell}(G,M_G^{G_{\mathfrak{P}}}  \mathcal{R}_{L_{\mathfrak{P}}}) \simeq \textrm{H}^{i}_{\ell}(G_{\mathfrak{P}},\mathcal{R}_{L_{\mathfrak{P}}})$$
$$\textrm{H}^{i}_{\ell}(G,\mathcal{U}_L^{\mathfrak{p}}) \simeq  \textrm{H}^{i}_{\ell}(G,M_G^{G_{\mathfrak{P}}}  \mathcal{U}_{L_{\mathfrak{P}}}) \simeq \textrm{H}^{i}_{\ell}(G_{\mathfrak{P}},\mathcal{U}_{L_{\mathfrak{P}}})$$
 Due to the choice of $S$, if $\mathfrak{p} \not\in S$ then $L_{\mathfrak{P}}/K_{\mathfrak{p}}$  is an unramified $\ell$-extension, hence  $ \textrm{H}^{i}_{\ell}(G_{\mathfrak{P}},\mathcal{U}_{L_{\mathfrak{P}}})=1$ by the next proposition.
\end{proof}
\smallskip

\begin{proposition}
Let  $L_{\mathfrak{P}}/K_{\mathfrak{p}}$ be an unramified $\ell$-extension then we have:
\begin{center}
$ \textrm{H}^{i}_{\ell}(\textrm{Gal}(L_{\mathfrak{P}}/K_{\mathfrak{p}}),{\mathcal{U}}_{L_{\mathfrak{P}}}))=1$
for $i=0, 1$.
\end{center}
\end{proposition}

\begin{proof}
The exact sequence: $ 1 \longrightarrow \mathcal{U}_{L_{\mathfrak{P}}} \longrightarrow \mathcal{R}_{L_{\mathfrak{P}}} \longrightarrow\mathbb{Z}_{\ell} \longrightarrow 1 $ induces a long sequence of  cohomology: 
$$ 1 \longrightarrow \mathcal{U}_{K_{\mathfrak{p}}} \longrightarrow \mathcal{R}_{K_{\mathfrak{p}}} \longrightarrow
\mathbb{Z}_{\ell} \longrightarrow \textrm{H}^{1}_{\ell}(\textrm{Gal}(L_{\mathfrak{P}}/K_{\mathfrak{p}}), \mathcal{U}_{L_{\mathfrak{P}}}) $$
where the map $  \mathcal{R}_{L_{\mathfrak{P}}} \longrightarrow \mathbb{Z}_{\ell}$ is the restriction of the valuation $v_{\mathfrak{P}} $. As $L_{\mathfrak{P}}/K_{\mathfrak{p}}$ is an unramified extension: $e_{L_{\mathfrak{P}}/K_{\mathfrak{p}}}=1$; this restriction is surjective so:  
\begin{center}
 $\textrm{H}^{1}_{\ell}(\textrm{Gal}(L_{\mathfrak{P}}/K_{\mathfrak{p}}),\mathcal{U}_{L_{\mathfrak{P}}}))=1. $ 
\end{center}
But due to the proof  p. $9$, we have
$$h_{\ell}(\textrm{Gal}(L_{\mathfrak{P}}/K_{\mathfrak{p}}),\mathcal{U}_{L_{\mathfrak{P}}}))=1 \quad
\textrm{thus} \quad \textrm{H}^{0}_{\ell}(\textrm{Gal}(L_{\mathfrak{P}}/K_{\mathfrak{p}}),\mathcal{U}_{L_{\mathfrak{P}}}))=1.$$
\end{proof}
\smallskip

\noindent{Consequently we obtain}
$$\textrm{H}^{i}_{\ell}(G,\mathcal{J}_L^{S}) \simeq \bigoplus_{\mathfrak{p} \in S} \textrm{H}^{i}_{\ell}(G_{\mathfrak{P}}, \mathcal{R}_{L_{\mathfrak{P}}}) $$
and 
$$ \mathrm{H}^{i}_{\ell}(G,\mathcal{J}_L)=\varprojlim_{S}  \textrm{H}^{i}_{\ell}(G,\mathcal{J}_L^S) =\varprojlim_{S} \bigoplus_{\mathfrak{p}} \mathrm{H}^{i}_{\ell}(G_{\mathfrak{P}},\mathcal{R}_{L_{\mathfrak{P}}})= \bigoplus_{\mathfrak{p}} \mathrm{H}^{i}_{\ell}(G_{\mathfrak{P}},\mathcal{R}_{L_{\mathfrak{P}}})$$
\medskip

\noindent\textbf{Step 3: }

\noindent{The $\ell$-group of $S$-units is $\mathcal{E}_{K}^{S}=\mathcal{R}_K \cap \mathcal{J}_K^S$. Let $S$ be a set of primes containing the infinite and the ramified primes, we show: }
\begin{center}
$h_{\ell}(G,\mathcal{E}_{L}^{S})= \dfrac{1}{\ell^{n}}  \prod_{\mathfrak{p} \in S} n_{\mathfrak{p}}, $
\end{center}
where $ n_{\mathfrak{p}}$ denotes the index of the decomposition sub-group. 
We are done as that the Herbrand quotient,  linked to a Galois module in a cyclic extension,  only depends to the character of the representation which is associated: it gives the structure of $G$-module up to a finite; and we use the property which says that if you consider a sub-module of finite index then its Herbrand quotient is trivial. This character is given by the Herbrand's representation character. 
\medskip

\noindent\textbf{Step 4: conclusion }

\noindent{Let $S$ be the set of primes described before, then we have: }
$$ 1 \longrightarrow \mathcal{E}_L^S \longrightarrow \mathcal{J}_L^S \longrightarrow \mathcal{J}_L^S \cdot \mathcal{R}_L /
\mathcal{R}_L= \mathcal{C}_L \longrightarrow 1.$$ 
As $L/K$ is a cyclic $\ell$-extension we get: $$ h_{\ell}(G,\mathcal{J}_L^S )=h_{\ell}(G, \mathcal{E}_K^S) \cdot h_{\ell}(G,\mathcal{C}_L).$$
But $$\textrm{H}^{i}_{\ell}(G,\mathcal{J}_L^{S}) \simeq \prod_{\mathfrak{p} \in S} \textrm{H}^{i}_{\ell} (G_{\mathfrak{P}}, \mathcal{R}_{L_{\mathfrak{P}}})$$
for $i=0,1$. From the local class field axiom we get:
$$| \textrm{H}^{0}_{\ell}(G_{\mathfrak{P}},\mathcal{R}_{L_{\mathfrak{P}}})|=n_{\mathfrak{p}} \quad \textrm{and} \quad
| \textrm{H}^{1}_{\ell}(G_{\mathfrak{P}},\mathcal{R}_{L_{\mathfrak{P}}})|=1$$
Thus, $ h_{\ell}(G,\mathcal{J}_L^S )=\prod_{\mathfrak{p} \in S} n_{\mathfrak{p}}$. 
By step 3: $h_{\ell}(G,\mathcal{E}_{L}^{S})= \frac{1}{\ell^{n}}  \prod_{\mathfrak{p} \in S} n_{\mathfrak{p}}$,  so $ h_{\ell}(G, \mathcal{C}_L)= \ell^{n}$
\end{proof}
\medskip

\subsection{The class field axiom}

This subsection is devoted to prove:

\begin{theorem}\textbf{The class field axiom}
Let $L/K$ be a cyclic $\ell$-extension of algebraic number fields then we have:
\begin{center}
$  |\textrm{H}^{i}_{\ell}(G(L / K),  \mathcal{C}_{L})|= \left\{
    \begin{array}{ll}
     [L:K] & \mbox{ for } i=0  \\
      1 & \mbox{ for } i=1
    \end{array}
\right.$
\end{center}
\end{theorem}

\begin{proof}

\noindent{Since $h_{\ell}(G(L/K),\mathcal{C}_L)=[L:K]= \ell^{n}$, it  suffices to show that $$\textrm{H}^{-1}_{\ell}(G(L/K),  \mathcal{C}_L)=\textrm{H}^{1}_{\ell}(G(L/K),\mathcal{C}_L)=1.$$ We do it by induction on $n$.}

($i$) If $n=0$ then $L=K$ and the result is true.

($ii$) If $n=1$ then $L/K$ is a cyclic extension of prime degree $\ell$. 

\noindent{The exact sequence $1 \longrightarrow \mathcal{R}_L  \longrightarrow \mathcal{J}_L  \longrightarrow \mathcal{C}_L$
leads to the Herbrand  hexagon: }
\begin{displaymath}
\xymatrix{
&\textrm{H}^{0}_{\ell}(G,\mathcal{R}_L) \ar [r] & \textrm{H}^{0}_{\ell}(G,\mathcal{J}_L)  \ar [dr]& \\
 \textrm{H}^{-1}_{\ell}(G,\mathcal{C}_L) \ar[ur] &  &  & \textrm{H}^{0}_{\ell}(G, \mathcal{C}_L)   \ar[dl] \\
&  \textrm{H}^{-1}_{\ell}(G,\mathcal{J}_L) \ar[ul] & \textrm{H}^{-1}_{\ell}(G,\mathcal{R}_L)\ar[l] & \\
}
\end{displaymath}
By prop.3.2.4 we have $\textrm{H}^{i}_{\ell}(G,\mathcal{J}_L^{S}) \simeq \prod_{\mathfrak{p} \in S} \textrm{H}^{i}_{\ell}(G_{\mathfrak{P}}, \mathcal{R}_{L_{\mathfrak{P}}})$.
By the local class field axiom (theorem 2.5.1), we deduce $ \textrm{H}^{1}_{\ell}(G,\mathcal{J}_{L})=1$. Thus it suffices to prove that the map from $ \textrm{H}^{0}_{\ell}(G,\mathcal{R}_L) $ to $ \textrm{H}^{0}_{\ell}(G,\mathcal{J}_L)$ is  injective: this follows from the $\ell$-adic Hasse norm theorem (theorem 3.3.2).
\smallskip

($iii$)If  $n>1$ then $\ell < \ell^{n}$, let $M/K$ be a sub-extension of $L/K$ of prime degree $\ell$.

We have
\begin{center}
$ 1 \longrightarrow \textrm{H}^{1}_{\ell}(G(M/K), \mathcal{C}_M) \longrightarrow  \textrm{H}^{1}_{\ell}(G(L/K), \mathcal{C}_L) \longrightarrow 
 \textrm{H}^{1}_{\ell}(G(L/M), \mathcal{C}_L)  $
\end{center}
Indeed, if $g$ is a normal subgroup of $G$, and $A$ a $G$-module, then the following sequence is exact:
$$0 \longrightarrow \textrm{H}^{1}(G/g,A^g) \longrightarrow \textrm{H}^{1}(G,A) \longrightarrow \textrm{H}^{1}(g,A) $$.

\noindent{By assumption $  \textrm{H}^{1}_{\ell}(G(M/K), \mathcal{C}_M) =1$ as $| G(M/K)|= \ell$, and
$  \textrm{H}^{1}_{\ell}(G(L/M), \mathcal{C}_L) =1$  as $| G(L/M)|= \ell^{n-1} < \ell^{n}$. It follows that $  \textrm{H}^{1}_{\ell}(G(L/K), \mathcal{C}_L) =1$.}
\end{proof}
\medskip

\begin{theorem} \textbf{(The $\ell$-adic Hasse Norm Theorem)}
If $L/K$ is a cyclic extension of prime degree $\ell$, an element of the $\ell$-group of principal ideles is a norm  from $L/K$ if and only if it is a norm everywhere locally, i.e a norm from each completion $ L_{\mathfrak{P}}/K_{\mathfrak{p}}$
where $ \mathfrak{P} \mid \mathfrak{p}$.
\end{theorem}

\begin{proof}
Let $x$ be a principal idele such that $x = N_{L/K}(y)$ where $y \in \mathcal{R}_{L}$. Since $\mathcal{R}_{L}$ injectes in $\mathcal{J}_L$, which  surjectes to $\mathcal{R}_{L_{\mathfrak{P}}}$ we deduce that $x$ is a norm everywhere locally.

\noindent Conversely assume $x \in \mathcal{R}_{K} $ and write down 
$ x = \bar{x} . y^{\ell}$, where $\bar{x}$ denotes the image of $x$ 
in $ {K^{\times}/K^{\times}}^{\ell}\simeq
\mathcal{R}_{K} / \mathcal{R}_{K}^{\ell}$. Since $L/K$ is a cyclic extensionof prime degree $\ell$, $y^{\ell}$ is a norm. Moreover, by hypothesis $x$ is a norm everywhere locally which means that each component  $\bar{x}_{\mathfrak{p}}$,  for all $\mathfrak{p}$, is a norm. Using the usual Hasse norm theorem we conclude that $x$ is a norm.
\end{proof}
\medskip

\subsection{$G$ and the $G$-module}

Let $G$ be the Galois group of the maximal abelian pro-$\ell$-extension of $\mathbb{Q}$. The $G$-module is  the union of the $\ell$-adic iddele class groups $\mathcal{C}_K$ where $K$ runs through the finite extensions of $K$:
$ \bigcup_{[K:\mathbb{Q}] < \infty} \mathcal{C}_K.$ and $ \mathcal{C}_L$ is a $\textrm{Gal(L/K)}$-module. 

\subsection{ $\textrm{deg} : G \mapsto \mathbb{Z}_{\ell}$}

We fix an isomorphism such that :
$ \textrm{Gal}(\tilde{\mathbb{Q}}/\mathbb{Q}) \simeq \mathbb{Z}_{\ell}.$ This allows to define :
\begin{center}
$\begin{array}{ccccc}
\textrm{deg}  & : & G= \textrm{Gal}(\mathbb{Q}^{ab}/\mathbb{Q})  & \to &  \mathbb{Z}_{\ell}  \\
& & \phi& \mapsto & \phi_{\mid_{\tilde{\mathbb{Q}}}}\\
\end{array}$
\end{center}
Let $K/\mathbb{Q}$  a finite extension, we define: $ f_K=[K \cap \tilde{\mathbb{Q}} : \mathbb{Q}]$ and we obtain, by analogy with the local case, a surjective homomorphism $ \textrm{deg}_{K}=\frac{1}{f_K} \cdot \textrm{deg}$ such that 
$ \textrm{deg}_{K} : \; G_K \longrightarrow \mathbb{Z}_{\ell}$.

\subsection{The valuation}

\begin{defi}
\noindent{Let $L/K$ be a finite and abelian $\ell$-extension}, we then define the map:
\begin{center}
$ [\; \cdot \;, L/K]=\prod_{\mathfrak{p}}  (\alpha_{\mathfrak{p}},L_{\mathfrak{p}}/K_{\mathfrak{p}})$
for $\alpha \in \mathcal{J}_K $
\end{center}
where  $ L_{\mathfrak{p}}$ denotes the completion of $K_{\mathfrak{p}}$  with respect to an arbitrary place $\mathfrak{P} \mid \mathfrak{p}$ and  $(\alpha_{\mathfrak{p}},L_{\mathfrak{p}}/K_{\mathfrak{p}})$  is the local symbol.

\end{defi}

\begin{proposition}
Let $L/K$ and $L'/K'$ be finite and abelian $\ell$-extensions of number fields such that $K \subseteq K'$ and $L \subseteq L'$, then the following diagram is commutative:
$$\begin{CD}
  \mathcal{J}_{ K'}@>{ [\; \cdot \;, L'/K']}>>  \mathrm{Gal}(L'/K') \\
@V{N_{K'/K}}VV  @VVV\\
\mathcal{J}_K  @>{ [\; \cdot \;, L/K]} >> \mathrm{Gal}(L/K)
\end{CD}$$
\end{proposition}
\smallskip

\begin{proof}
Take $\alpha=(\alpha_{\mathfrak{P}}) \in \mathcal{J}_{K'}$. We get for $ \mathfrak{P} \mid \mathfrak{p}$: 
$(\alpha_{\mathfrak{P}},L^{'}_{\mathfrak{P}}/K^{'}_{\mathfrak{P}})_{\mid L_{\mathfrak{p}}}= N_{ K^{'}_{\mathfrak{P}} /K_{\mathfrak{p}}}(\alpha_{\mathfrak{P}}), L_{\mathfrak{p}}/K_{\mathfrak{p}})$ and
$$[N_{K'/K}(\alpha),L/K]= \prod_{\mathfrak{p}} (N_{K'/K}(\alpha)_{\mathfrak{p}}, L_{\mathfrak{p}}/K_{\mathfrak{p}})=
\prod_{\mathfrak{p}} \prod_{\mathfrak{P} \mid \mathfrak{p}} N_{ K^{'}_{\mathfrak{P}} /K_{\mathfrak{p}}}(\alpha_{\mathfrak{P}}) $$  so $$[N_{K'/K}(\alpha),L/K]  = \prod_{\mathfrak{P}} (\alpha_{\mathfrak{P}},L^{'}_{\mathfrak{P}}/K^{'}_{\mathfrak{P}})_{ / L}=[\alpha,L'/K']_{\mid L}.$$
 
\end{proof}

\begin{proposition}
For all roots of units $\zeta$  and for all $a \in \mathcal{R}_K$ we have
$$[a,(K(\zeta)/K)_{\ell}]=1$$  where $ (K(\zeta)/K)_\ell$ denotes the projection on the $\ell$- Sylow sub-group of $\textrm{Gal}(K(\zeta)/K).$
\end{proposition}

\begin{proof}
We follow~\cite[prop 6.3, p.~92]{Ne1}. By the previous proposition: $[N_{K/\mathbb{Q}}(a), (\mathbb{Q}(\zeta)/\mathbb{Q})_{\ell}]=[a,(K(\zeta)/K)_{\ell}]_{\mid \mathbb{Q}(\zeta)}$. Consequently it suffices to show the property for $K=\mathbb{Q}$. But
$$[a,(\mathbb{Q}(\zeta)/\mathbb{Q})_{\ell}] \zeta = \prod_{\mathfrak{p}}  (a,(\mathbb{Q}_p (\zeta)/\mathbb{Q}_p))_{\ell}.$$

\noindent{Let $q$ be a prime and $\zeta$  be a $q^m$-root of unity, with $q^m \ne 2$. We take $a \in \mathcal{R}_{\mathbb{Q}_{p}}$ and write $ a= u_p \cdot p^{v_p(a)}$ where $v_p$ is the usual normalized valuation on $\mathbb{Q}_{p}$. For $p \ne q$ and $ p \ne \infty$ the extension $\mathbb{Q}_{p}(\zeta)/ \mathbb{Q}_{p}$ is an unramified extension.} The fundamental principle  ~\cite[theorem 2.6, p.~25]{Ne1}  states that the local symbol associates the uniformising element to the Frobenius, one gets that $(p, (\mathbb{Q}_{p}(\zeta)/ \mathbb{Q}_{p}))_{\ell}$ corresponds to the Frobenius automorphism $\phi_{p}: \zeta \longrightarrow \zeta^{p}$. Moreover the following diagram is commutative: 
 $$\begin{CD}
  K_{\mathfrak{ p}}^{\times}@>{ (\; \cdot \; \dot,\textrm{Gal}(L_{\mathfrak{P}}/K_{\mathfrak{p}})) }>>  \textrm{Gal}(L_{\mathfrak{P}}/K_{\mathfrak{p}}) \\
@VVV  @VVV\\
\mathcal{R}_{K_{\mathfrak{p}}}  @>{  (\; \cdot \; \dot ,\textrm{Gal}(L_{\mathfrak{P}}/K_{\mathfrak{p}}))_{\ell} } >>\textrm{Gal}(L_{\mathfrak{P}}/K_{\mathfrak{p}})_{\ell}
\end{CD}$$
\noindent where the symbol on the top is the usual local symbol, and the symbol on the bottom is the $\ell$-adic local symbol.
Consequently, one deduces
$$(a, (\mathbb{Q}_{p}(\zeta)/ \mathbb{Q}_{p})_{\ell}) \zeta=\zeta^{n_p}$$
with
\begin{equation*}
  n_p=
     \begin{cases}
        p^{v_p(a)} & \text{for $p\ne q$ et $ p\ne\infty $} \\
        u_p^{-1}  & \text{for $ p=q$} \\
       sgn(a) & \text{ for $p = \infty $ }
     \end{cases}
\end{equation*}
So
$$[a,(\mathbb{Q}(\zeta)/\mathbb{Q})_{\ell}] \zeta = \prod_{\mathfrak{p}}  (a,(\mathbb{Q}_p (\zeta)/\mathbb{Q}_p)_{\ell})=\zeta^{\alpha}$$
And due to the product formula, $\alpha= \prod_{p} n_p= sgn(a) \cdot \prod_{p \ne \infty } p^{v_p(a)}  \cdot a^{-1}=1$.
\end{proof}

\begin{defi}
We define the valuation $v_K : \mathcal{C}_K \longrightarrow \mathbb{Z}_{\ell}$ as follows: 

$$\begin{CD}
\mathcal{C}_K @>{[\; \cdot \; , \tilde{K}/K]}>> G( \tilde{K}/K) @>{ deg_{K}}>>\mathbb{Z}_{\ell}
\end{CD}$$

\end{defi}

\begin{lemma}
 $v_K$ is well defined.
\end{lemma}

\begin{proof}
We show that $ \forall a \in \mathcal{R}_K, \, [a, \tilde{K}/K]=1.$ As $\tilde{K}/K$ is contained in the extension of $K$ obtained by adjoining roots of units it sufficies to show that, for $a \in \mathcal{R}_K$  and $\zeta$ a root of unit, $[a, (K(\zeta)/K)_\ell]=1$, this is proposition 3.6.2. Thus we deduce that  $ \mathcal{R}_K \subseteq  \textrm{Ker}([\; \cdot \;, \tilde{K}/K])$ .
\end{proof}
 
\begin{lemma}
$v_K$ is surjective and $[\mathcal{C}_K,\textrm{Gal}(\tilde{K}/K)]$ is closed in\ $\textrm{Gal}(\tilde{K}/K)$.
\end{lemma}

\begin{proof}
We follow ~\cite[prop 6.4, p.~93]{Ne1} . The local symbol is surjective, $[\mathcal{J}_K,\textrm{Gal}(L/K)]$ contains all decomposition groups  $\textrm{Gal}(L_{\mathfrak{P}}/K_{\mathfrak{p}})$. Thus all $\mathfrak{p}$ splits completely in the fixed field $M$ of $[\mathcal{J}_K,\textrm{Gal}(L/K)]$. This implies $M=K$ ans so $[\mathcal{J}_K,\textrm{Gal}(L/K)] = \textrm{Gal}(L/K)$ and that $[\mathcal{J}_K,\textrm{Gal}(\tilde{K}/K)]$. This yields furthermore that $ [\mathcal{J}_K,\textrm{Gal}(\tilde{K}/K)=[\mathcal{C}_K,\textrm{Gal}(\tilde{K}/K)]$ is dense in $\textrm{Gal}(\tilde{K}/K)$.
\end{proof}

\begin{lemma}
$[\mathcal{C}_K,\textrm{Gal}(\tilde{K}/K)]$ is dense in\ $\textrm{Gal}(\tilde{K}/K)$.
\end{lemma}

\begin{proof}
We have  $ [\mathcal{J}_K,\textrm{Gal}(\tilde{K}/K)=[\mathcal{C}_K,\textrm{Gal}(\tilde{K}/K)]$ as $ [\mathcal{R}_K,\textrm{Gal}(\tilde{K}/K)=1$. Let $\textrm{Gal}(\tilde{K}/L)$ be a neighborhood of the neutral in $\textrm{Gal}(\tilde{K}/K)$, where $L$ is a finite Galois extension of $K$ of degree $\ell^{n}$.  As $\mathcal{J}_K= \mathcal{U}_K \times \oplus \pi_{\mathfrak{p}}^{\mathbb{Z}_{\ell}}$
where $\mathcal{U}_K$ is the $\ell$-adic group of units,  a  neighborhood of the neutral is of the shape:
$ \mathcal{U}_K^{'} \times \oplus \pi_{\mathfrak{p}}^{ \ell^{k_{\mathfrak{p}}} \mathbb{Z}_{\ell}}$
where $\mathcal{U}_K^{'}$ is an open submodule of $\mathcal{U}_K$ and $k_{\mathfrak{p}}$ an integer. We can choose  $k_{\mathfrak{p}} >n$. Thus the image of $\pi_{\mathfrak{p}}^{\ell^{ k_{\mathfrak{p}}} \mathbb{Z}_{\ell}}$ is trivial through the local symbol. Moreover if $\mathfrak{p} \mid \ell $ then the local extension is unramified and the image of an element of $\mathcal{U}^{'}_K$ is trivial. If $\mathfrak{p} \nmid \ell$  then  thanks to the filtration of the group of units we can obtain a trivial image. Therefore the map $[\; \cdot \;,\textrm{Gal}(\tilde{K}/K)] : \mathcal{J}_K \mapsto \textrm{Gal}(\tilde{K}/K) $ is continuous
and as $\mathcal{C}_{K}$ is compact, we deduce that $[\mathcal{C}_K,\textrm{Gal}(\tilde{K}/K)]$ is dense in\ $\textrm{Gal}(\tilde{K}/K)$.
\end{proof}

\begin{lemma}
 $v_K$ is henselian with respect to $\textrm{deg}$.
\end{lemma}

\begin{proof}
We have: $$v_K(N_{L/K} \mathcal{C}_L)= v_K(N_{L/K} \mathcal{J}_L)= deg_K \circ [ N_{L/K}\mathcal{J}_L, \tilde{K}/K]$$
(as $[\mathcal{R}_K, \, \textrm{Gal}(\tilde{K}/K)]=1$). Moreover $ deg_K=\frac{1}{f_K} \cdot deg $ and $f_{L/K}=f_L/f_K$  that is why $ deg_K=f_{L/K} \cdot deg_L$. By  proposition 3.6.1, the diagram is commutative: 
$$\begin{CD}
  \mathcal{J}_{ L'}@>{ [\; \cdot \; \, \tilde{L}/L]}>>  \textrm{Gal}(\tilde{L}/L) \\
@V{N_{L/K}}VV  @VVV\\
\mathcal{J}_K  @>{ [\; \cdot \; \, \tilde{K}/K]} >>\textrm{Gal}(\tilde{K}/K)
\end{CD}$$
consequently $[N_{L/K} \mathcal{J}_L, \tilde{K}/K]=[\mathcal{J}_L, \tilde{L}/L]$
thus we deduce, by the surjectivity of $v_L$ that
$$ v_K(N_{L/K} \mathcal{C}_L)= f_{L/K} \cdot deg_L \circ  [\mathcal{J}_L,\tilde{L}/L]=f_{L/K} \cdot v_L(\mathcal{C}_L)=f_{L/K} \cdot \mathbb{Z}_{\ell}$$
\end{proof}

\begin{cor}
 $v_K$ is well defined and both surjective and  henselian with respect to $\textrm{deg}$.
\end{cor}

\begin{cor}
$(\textrm{deg},v)$  is a class field pair, and  $A_K:=\mathcal{C}_K$  satisfies the class field axiom.
Thus for all Galois $\ell$-extension of a number field $K$  we get an isomorphism: 
$$ \textrm{Gal}( L/K)^{ab} \simeq \mathcal{C}_{K}/N_{ L/K}
\mathcal{C}_L.$$
In particular, we get a one to one correspondence between finite and abelian $\ell$-extensions of a number field $K$ and 
open subgroups of  $\mathcal{C}_K$.
\end{cor}
\medskip

\noindent{ \sc Acknowledgements}\smallskip

\noindent{I would like to thank Boas Erez for proposing me this question, my advisor Jean-François Jaulent  for our dicussions and 
Karim Belabas for his helpful comments on earlier versions of this article.}

{\small

\noindent{Univ. Bordeaux,
IMB, UMR 5251,
351 Cours de la Libération,
F-33400 Talence, France}

\noindent{CNRS,
IMB, UMR 5251,
351 Cours de la Libération,
F-33400 Talence, France}
}
\smallskip

\noindent{\small{ Email address \smallskip
\tt stephanie.reglade@math.u-bordeaux1.fr
}}

\end{document}